\documentclass[11pt,letterpaper]{article}

\usepackage{enumerate, framed, amsmath, amssymb, amsthm, tikz, pgfplots, url}
\usepackage{multicol, float, color, chngcntr}
\counterwithin{figure}{section}

\newcommand{\ZZ}{\mathbb{Z}}

\newcommand{\FF}{\mathbb{F}}

\newcommand{\trm}[1]{\textrm{#1}}

\newtheorem{theorem}{Theorem}[section]
\newtheorem{lemma}[theorem]{Lemma}

\date{}
\title{Analogues of the $3x+1$ Problem in Polynomial Rings of Characteristic 2}
\author{Daniel Nichols}

\begin{document}

\maketitle

\begin{abstract}
The Collatz conjecture (also known as the $3x+1$ problem) concerns the behavior of the discrete dynamical system on the positive integers defined by iteration of the so-called $3x+1$ function. We investigate analogous dynamical systems in rings of functions of algebraic curves over $\FF_2$. We prove in this setting a generalized analogue of a theorem of Terras concerning the asymptotic distribution of stopping times. We also present experimental data on the behavior of these dynamical systems.
\end{abstract}

\noindent This is a preprint of an article published by Taylor \& Francis in \textit{Experimental Mathematics} on October 7, 2016, available online: \url{http://www.tandfonline.com/doi/full/10.1080/10586458.2016.1227734}.

\section{Introduction}

The dynamical system on the positive integers defined by the $3x+1$ map $T:\ZZ \rightarrow \ZZ$ can be modelled by a one-dimensional random walk, as described in \cite{lagarias-models}. We can write
	$$ \log_2 T^N(x) \approx \log_2 x - N + b_3\sum_{k=0}^N X_k, $$
where $b_3$ is a constant and the $X_k$ are IID (independent identically distributed) Bernoulli random variables. That is, $X_k$ takes values in $\left\{ 0, 1 \right\}$ each with probability $1/2$. The most well-known unsolved problem concerning this dynamical system is the Collatz conjecture, which states that all trajectories eventually reach 1.

Models of this form can also be used for the more general $mx+1$ problem, where $m > 1$ is an odd positive integer, simply by substituting a different constant $b_m$ in place of $b_3$. Specifically, we define $b_m = \log_2 (m + 2/3)$. When $m = 3$, this model predicts that almost every positive integer $x$ has finite stopping time. However, for all odd $m > 3$, it predicts that a significant number of trajectories have infinite stopping time. The statistical tendency of such a random walk to diverge is entirely determined by value of $b_m$. As $m$ increases, the probability of divergence in the $mx+1$ system quickly approaches 1.

In this paper we discuss a class of similar $mx+1$ dynamical systems in $\FF_2[t]$ which can also be modeled by a random walk. Like those who have previously studied these systems, we are motivated by the principle that problems concerning $\FF_p[t]$ are often easier to solve than the corresponding problems in $\ZZ$, since we can often exploit the rich algebraic structure of polynomial rings over a field to simplify both numerical computations and theoretical analysis. The random walk model for $mx+1$ systems turns out to be even more accurate in $\FF_2[t]$ than in the integer case, and the parameter $b_m$ is always an integer. We show that in a certain sense the random walks associated to these polynomial $mx+1$ problems interpolate between those of the traditional $mx+1$ problems in $\ZZ$, providing examples of a more general class of pseudo-random dynamical systems.

From algebraic geometry we know that $\FF_p[t]$ is the ring of regular functions of the affine line over $\FF_p$.  This connection and the rich geometric tools available are the reasons why many arithmetic problems over $\ZZ$ become more approachable when we work over $\FF_p[t]$. Once we bring in this geometric picture, it is natural to go beyond the affine line and ask whether there is a way to define a more general type of $mx + 1$ system on other algebraic curves over $\FF_2$. We construct such systems for curves of the form $x^2 + tx + r(t) = 0$, where $r(t)$ is an irreducible polynomial over $\FF_2$. (The linear term $xt$ is necessary in order to define a smooth affine hyperelliptic curve in characteristic 2.) Since this is a family of hyperelliptic curves, the genus and other geometric properties are well-understood.

For these new $mx+1$ systems, the random walk model is somewhat different from the one for $\FF_2[t]$. Instead of moving left or right with equal probability (i.e. a `coin flip'), we use a random walk with unequal probabilities. While this model is not as directly comparable to the classical $3x+1$ random walk model, it does provide an interesting generalization and allows us to prove some useful results.

Figure \ref{fig-randomwalk-line} below shows the random walks associated to $mx+1$ problems in both $\ZZ$ and $\FF_2[t]$, organized by the value of $b_m$. Towards the left side of the scale ($b_m \leq 1$), trajectories are very likely to converge to one. In fact, Hicks et al. \cite{hicks-polynomial-analogue} were able to prove the analogue of the Collatz conjecture for the case $m = t + 1$. Towards the right side of the scale ($b_m > 2$), trajectories exhibit an increasingly strong tendency to diverge. We prove that for all $m$ of degree at least 3, there is a nonzero probability that a randomly chosen polynomial will have infinite stopping time. This means that the Collatz conjecture must be false when $\deg m > 3$. However, trajectories with infinite stopping time do not necessarily diverge, so the existence of true divergent trajectories is still an open question for most of these polynomials.

This leaves the two quadratic odd polynomials $t^2 + 1$ and $t^2 + t + 1$ as the most interesting cases. For the first of these, Matthews et al. \cite{matthews-syracuse} showed that the trajectory of a certain constructed polynomial must diverge. For the second, we observed empirically many nontrivial cyclic orbits of the $mx+1$ function. So the Collatz analogue is disproved in these cases as well.

The original $3x+1$ problem in $\ZZ$ lies in the interesting middle area $1 < b_m < 2$, where the asymptotic properties of the random walk are least predictable. About these systems it is difficult to prove anything at all.

\begin{figure*}[h]\label{fig-randomwalk-line}
\centering
\caption{Tendency of random walk models associated with different $mx+1$ problems}
	\begin{tikzpicture}[scale=1]
		\draw [<->, very thick] (-0.5,0) -- (8.5,0);
		\node [right] at (8.5,0) {$b_m$};
		
		\draw [very thick] (0,-0.3) -- (0,0.3);
		\node [above] at (0,0.3) {$0$};
		\draw [very thick] (4,-0.3) -- (4,0.3);
		\node [above] at (4,0.3) {$2$};
		\draw [very thick] (8,-0.3) -- (8,0.3);
		\node [above] at (8,0.3) {$4$};
		
		\draw [thick] (2,-0.2) -- (2,0.2);
		\node [above] at (2,0.2) {\small $1$};
		\draw [thick] (6,-0.2) -- (6,0.2);
		\node [above] at (6,0.2) {\small $3$};

		\node [below] at (-1,2) {$\ZZ$:};
		
		\draw [<-, dashed] (1.3,0) -- (1.3,1.4);
		\node [below] at (1.3,0) {\scriptsize $0.74$};
		\node [above, thick, draw=black] (zm1) at (1.3,1.4) {\small $1x+1$};
		\node [above, thick, draw=black] (zm1proof) at (0.6,2.2) {\scriptsize Collatz proved};
		\draw [->, thick] (zm1proof) -- (zm1);
		
		\draw [<-, dashed] (3.75,0) -- (3.75,1.4);
		\node [above, thick, draw=black] (zm3) at (3.6,1.4) {\small $3x+1$};
		\node [below] at (3.7,0) {\scriptsize $1.87$};
		
		\draw [<-, dashed] (5.00,0) -- (5.00,1.4);
		\node [above, thick, draw=black] (zm5) at (5.00,1.4) {\small $5x+1$};
		\node [below] at (4.8,0) {\scriptsize $2.50$};
		
		\draw [<-, dashed] (5.88,0) -- (5.88,1.4);
		\node [above, thick, draw=black] (zm7) at (6.4,1.4) {\small $7x+1$};
		\node [below] at (5.65,0) {\scriptsize $2.94$};
		
		\node [below] at (-1,-2) {$\FF_2[t]$:};
		
		\draw [<-, dashed] (2,-0.2) -- (2,-1.4);
		\node [below, thick, draw=black] (f2tm1) at (2,-1.4) {\scriptsize $\deg m = 1$};
		\node [below, thick, draw=black] (f2tm1proof) at (1.5,-2.4) {\scriptsize Collatz proved};
		\draw [->, thick] (f2tm1proof) -- (f2tm1);
		
		\draw [<-, dashed] (4,-0.3) -- (4,-1.4);
		\node [below, thick, draw=black] (f2tm2) at (4,-1.4) {\scriptsize $\deg m = 2$};
		
		\draw [<-, dashed] (6,-0.2) -- (6,-1.4);
		\node [below, thick, draw=black] (f2tm3) at (6,-1.4) {\scriptsize $\deg m = 3$};

		\node [below, thick, draw=black] (f2tm23proof) at (4.5,-2.4) {\scriptsize Divergent $f$ exists*};
		\draw [->, dashed, thick] (f2tm23proof) -- (f2tm2);
		\draw [->, dashed, thick] (f2tm23proof) -- (f2tm3);
		
		\node [below, thick, draw=black] (f2tm3proof) at (7.3,-2.4) {\scriptsize $P(\sigma = \infty) > 0$};
		\draw [->, thick] (f2tm3proof) -- (f2tm3);
		
	\end{tikzpicture}
\end{figure*}
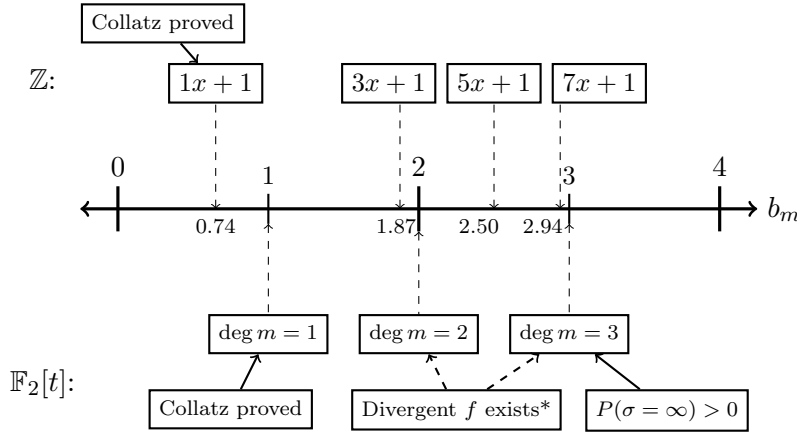

Following this introduction, we first examine $mx+1$ systems in $\FF_2[t]$. After providing some definitions and summarizing past work in this area, we prove a stronger analogue of Terras' theorem on the probability of infinite stopping times. We then present some experimental data on these systems concerning stopping times, cycle lengths, and rate of growth for seemingly-divergent trajectories.

In the second part of the paper, we define a family of similar $mx+1$ dynamical systems over algebraic curves of the form $x^2 + tx + r(t)$. In this case we need to use a more intricate random walk model featuring a Bernoulli random variable with $p=1/4$ instead of $1/2$. We prove a stronger analogue of Terras' theorem in this setting also, and present some experimental data on stopping times.

This work was supported in part by NSA grant H98230-14-1-0307. Professors Carl Pomerance and Jeffrey Lagarias provided helpful comments on an earlier draft of this paper, for which we are very grateful. We also owe thanks to Professor Hans Johnston for his help with our computations. This paper represents one part of the author's dissertation, supervised and guided by Professor Siman Wong.

\section{$3x+1$ analogue in the Ring $\FF_2[t]$}

Let $m \in \FF_2[t]$ be a fixed \textbf{odd} polynomial (meaning $m(0) = 1$). The $mx+1$ map $T : \FF_2[t] \rightarrow \FF_2[t]$ is defined by the formula
\[ T(f) = \left\{ \begin{array}{c l}
\frac{f}{t},& f \equiv 0 \mod{t} \\[0.1in]
\frac{mf+1}{t},& f \not\equiv 0 \mod{t},
\end{array} \right. \]
Iteration of this function defines a discrete dynamical system on $\FF_2[t]$. Given a starting element $f \in \FF_2[t]$, we call the sequence $ \left\{ f, T(f), T^2(f), T^3(f), \ldots \right\}$ the \textbf{trajectory} of $x$. Each trajectory must either become cyclic at some point or else diverge, meaning
$$\lim_{k \rightarrow \infty} \deg T^k(f) = \infty .$$

The \textbf{Collatz conjecture} states that every trajectory of the $3x+1$ dynamical system in $\ZZ$ eventually reaches 1. This implies (among other things) that the only cycle is $\left\{ 1, 2 \right\}$. When we view each element of $\FF_2[t]$ as sequence of binary coefficients, there is a natural set bijection between $\FF_2[t]$ and the ring of nonnegative integers with binary representation. In that sense, the $mx+1$ system in $\FF_2[t]$ can be viewed as a dynamical system on the positive integers similar to the one defined by the original $3x+1$ function. It is natural to consider the analogue of the Collatz conjecture in this setting. That is, for a given polynomial $m \in \FF_2[t]$, does every $mx+1$ trajectory in $\FF_2[t]$ eventually reach 1?

Hicks, Mullen, Yucas, and Zavislak \cite{hicks-polynomial-analogue} were able to prove that for $m = t + 1$, all sequences eventually reach 1. Therefore the conjecture is true when $m = t + 1$. However, for most choices of $m$ we can easily find nontrivial cycles. For example, when $m = t^2 + t + 1$, the trajectory of $f = t^2 + 1$ is
\[ \begin{array}{|c|c|}
\hline
k & T^k(f) \\ \hline
0	&	t^2 + 1	\\
1	&	t^3 + t^2 + 1 \\
2	&	t^4	+ 1 \\
3	&	t^5 + t^4 + t^3 + t + 1	\\
4	&	t^6 + t^4	\\
5	&	t^5 + t^3 	\\
6	&	t^4 + t^2	\\
7	&	t^3 + t	\\
8	&	t^2	+ 1 \\ \hline
\end{array} \]
This sequence repeats with period 8. The existence of this cycle disproves the Collatz conjecture analogue for $m = t^2 + t + 1$.

There are also trajectories which seem very likely to diverge. The trajectory of $f = t^6 + t^2 + t + 1$ does not repeat a value within the first two billion iterations. Figure \ref{fig-f2t-trajectories} in section \ref{sec-f2t-experimental} shows a plot of this trajectory, along with two others that seem to diverge. Matthews and Leigh \cite{matthews-syracuse} were able to exhibit a polynomial with a provably divergent trajectory when $m = t^2 + 1$, and it is easy to apply their construction to all $m$ of the form $t^n + 1$ for even $n \geq 2$. Experimental data confirms our expectation that a higher-degree polynomial $m$ causes a higher rate of apparently-divergent trajectories.

We want to understand the dynamics of $mx+1$ for a given polynomial $m$. Since the Collatz conjecture analogue is likely false for $\deg m > 1$, we instead consider the following two questions:
\begin{enumerate}
	\item Do divergent $mx+1$ trajectories exist? If so, what is the density of divergent trajectories in the set of all trajectories?
	\item Do cyclic trajectories exist? If so, how are cycle lengths distributed?
\end{enumerate}

In order to investigate the first of these questions, we define the \textbf{stopping time} $\sigma(f)$ to be the minimum number of steps before the trajectory of $f$ reaches a polynomial of lower degree than $f$. That is,
$$ \sigma(f) = \inf \left\{ k > 0 : \deg T^k(f) < \deg f \right\} .$$
Note that if $m$ is even (i.e. $m(0) = 0$) then necessarily $\sigma(f) = 1$. If the trajectory of $f$ never reaches a polynomial of lower degree, we set $\sigma(f) = \infty$. Clearly if $\sigma(f) < \infty$ for all $f$, then the Collatz conjecture analogue must be true. On the other hand, if there exists any $f$ with $\sigma(f) = \infty$, then the trajectory of $f$ cannot reach 1 and so the conjecture must be false.

For the integer $3x+1$ problem, Terras \cite{terras} proved the following theorem concerning stopping times. An alternative proof was given soon afterwards by Everett \cite{everett}.
\begin{theorem}\label{th-terras}
	Almost every positive integer has finite $3x+1$ stopping time. That is,
	$$ \lim_{N \rightarrow \infty} P \left( \sigma(x) < \infty \,|\, 0 < x \leq N \right) = 1.$$
\end{theorem}

Everett's proof proceeds by showing that $3x+1$ trajectories are closely modeled by a one-dimensional random walk, and then using the statistical properties of this model. We use a similar method to prove a stronger version of this theorem for $mx+1$ systems in $\FF_2[t]$. Our theorem is stronger in that it gives precise predictions for the density of divergent trajectories.

\begin{theorem} \label{th-f2t-terras}
Let $m \in \FF_2[t]$ with $\deg m = d$ and let $P_m$ be the asymptotic probability that a randomly chosen polynomial in $\FF_2[t]$ has finite $mx+1$ stopping time. That is,
$$ P_m = \lim_{N \rightarrow \infty} P \left( \sigma(f) < \infty | \deg f < N \right) .$$
If $d \leq 2$, then $P_m = 1$. If $d > 2$, then $P_m \in (1/2,1)$ is the unique real root of the polynomial $g_d(z) = z^d - 2z + 1$ inside the unit disk.
\end{theorem}

\subsection{Proof of Terras' theorem analogue in $\FF_2[t]$}
\label{sec-proof-f2t}

First, we define the \textbf{parity sequence} of $f$ to be $\left\{ p_0, p_1, p_2, \ldots \right\}$, where $p_k = (T^k(f))(0)$. That is, $p_k$ is the constant term of $T^k(f)$, which indicates whether $tT^{k+1}(f) = mT^k(f) + 1$ or $tT^{k+1}(f) = T^k(f)$. To prove Theorem \ref{th-f2t-terras}, we follow the outline used by Everett \cite{everett} to prove the corresponding result for the $3x+1$ system in $\ZZ$. We prove that the parity sequence of a uniformly-chosen polynomial in $\FF_2[t]$ is uniformly distributed in the set of sequences in $\left\{0,1\right\}$. Then we prove that almost all such sequences correspond to polynomials with finite stopping time.

If we want to find the first $N$ terms of the parity sequence of a polynomial $f \in \FF_2[t]$, we only need to consider the lowest $N$ coefficients of $f$. The higher coefficients will have no effect until later in the sequence. In fact, the parity sequences of all polynomials in the set $\left\{ g + t^N q : q \in \FF_2[t] \right\}$ must have the same first $N$ terms. Therefore, there is a well-defined set function
$$\Phi_m : \FF_2[t]/t^N \longrightarrow \left\{ 0, 1 \right\}^N $$
which maps each element of $\FF_2[t]/t^N$ to the first $N$ terms of its $m$-parity sequence. We claim that this function is one-to-one.
\begin{lemma} \label{l-f2t-phi}
The map $\Phi_m$ described above is a set bijection. That is, every sequence $\left\{ p_0, p_1, \ldots, p_{N-1} \right\}$ with $p_i \in \left\{ 0, 1 \right\}$ is the first $N$ terms of the parity sequence of a unique polynomial $f \in \FF_2[t]$ with $ \deg f < N$. Specifically, the parity sequence determines the initial polynomial $f$ and its $N$-th iterate $T^N(f)$ as follows, up to choice of $q_N$:
\begin{align*}
f &= g_{N-1} + t^N q_N ,& \deg g_{N-1} < N \\
T^N(f) &= h_{N-1} + m^{s(N)} q_N ,& \deg h_{N-1} < ds(N)
\end{align*}
where $d = \deg m$ and $s(N) = \sum_{i=0}^{N-1} p_i$. Therefore, parity sequences of polynomials in $\FF_2[t]$ of degree $< N$ are distributed uniformly in $\left\{0,1\right\}^N$.
\end{lemma}
Note that $s(N)$ is just the number of $1$'s which appear in the first $N$ terms of the parity sequence of $f$, which is the number of multiplications that occur in the first $N$ steps of the trajectory of $f$.

First, an informal explanation. Suppose we know the first term $p_0$ of the parity sequence of $f$. Using this, we can determine whether $f$ is `odd' or `even'. That is, we can find $f$ modulo $t$. If we also know $p_1$, we can `lift' our knowledge of $f$, obtaining $f$ modulo $t^2$. We also learn the parity of $f_1$. If we know $p_2$, we gain one more degree of precision in $f$ and $T(f)$, and additionally we learn the parity of $T^2(f)$. More generally, if we know $f$ modulo $t^{k+1}$ and we know $p_k$, we can perform a sort of lift and find the value of $f$ modulo $t^{k+2}$, and we also learn a bit more about $f_{k+1}$. In effect, there is an algorithm which constructs the unique polynomial of degree $< N$ with a given parity sequence $\left\{ p_0, p_1, \ldots, p_{N-1} \right\}$. To prove the theorem, we just need to describe this algorithm and verify that it works.

\begin{proof}
We proceed by induction on $N = 1,2,\ldots$. The base case is $N = 1$. If $p_0 = 0$, then $f = tq_1$ and so $T(f) = q_1$. If $p_0 = 1$, then $f = 1 + tq_1$ and $T(f) = (m+1)/t + mq_1$.

Now assume the theorem is true for all values up to $N$. We argue that it is true for $N+1$.  There are four cases to consider, depending on the values of $h_{N-1}(0)$ and $p_N$ in $\left\{0,1\right\}$. For instance, suppose $h_{N-1}(0) = p_N = 0$. That is, the $N$-th term of the trajectory is `even' and $q_N$ is also even. Let $q_{N} = tq_{N+1}$. Then the next term is
\begin{align*}
f_{N+1} &= \frac{ f_N }{ t } = \frac{ h_{N-1} + m^{s(N)}q_N }{ t } \\
&= \frac{ h_{N-1} }{ t } + m^{s(N+1)} q_{N+1}
\end{align*}
We can rewrite the initial polynomial as
$$ f = g_{N-1} + t^{N+1}q_{N+1} .$$
Since $\deg h_{N-1}/t < ds(N)$ and $\deg g_{N-1} < N+1$, the theorem holds in this case. The other three cases are extremely similar.\footnote{A complete proof of all four cases is given in a supplemental document available on the author's website.}

This proves that $f$ modulo $t^N$ together with the parity sequence term $p_N$ is sufficient to uniquely identify $f$ modulo $t^{N+1}$. Therefore, every length-$N$ parity sequence must arise from some polynomial in $\FF_2[t]/t^N$. There are $2^N$ polynomials of degree $< N$, and there are $2^N$ binary sequences of length $N$. So by cardinality, the surjective map $\Phi_m : \FF_2[t]/t^N \rightarrow \left\{ 0,1 \right\}^N$ is a set bijection.
\end{proof}

We have shown that the parity sequence of a randomly chosen polynomial $f \in \FF_2[t]$ of degree less than $M$ is distributed uniformly in $\left\{0,1\right\}^N$. Now we describe how the parity sequence of $f$ determines the degree of $T^N(f)$. If the parity sequence of $f$ is $\left\{ p_k \right\}$, then
$$ \deg T^N(f) = \deg f - N + d\sum_{k=0}^{N-1} p_k .$$

Since $\left\{ p_k \right\}_{k=0}^{N-1}$ is uniformly distributed in $\left\{0,1\right\}^N$, we can write
$$ \deg T^N(f) - \deg f \approx d\sum_{k=0}^{N-1}X_k - N $$
where $X_k$ are IID uniform Bernoulli random variables. This leads immediately to the following theorem:
\begin{theorem}\label{th-f2t-sigma-m}
The probability that a randomly chosen $f \in \FF_2[t]$ has finite $mx+1$ stopping time is
\begin{align} \label{p-sigma-inf} P(\sigma(f) < \infty) = P \left( \exists N > 0 : \sum_{k=0}^{N-1} X_k < \frac{1}{d}N \right)
\end{align}
where $X_i$ are IID uniform Bernoulli random variables and $d = \deg m$.
\end{theorem}

We will now show that this probability is the root of a certain simple polynomial which depends only on $d = \deg m$, thus proving Theorem \ref{th-f2t-terras}.
\begin{lemma} \label{l-f2t-ruin}
For $k = 0, \ldots, N-1$, let $X_k$ be IID uniform Bernoulli variables and let $P_d$ be defined
$$ P_d = P \left( \exists N > 0 : \sum_{k=0}^{N-1} X_k < \frac{1}{d}N \right) .$$
Then $P_1 = P_2 = 1$, and for $d > 2$, $P_d$ is the unique real root of the polynomial $g_d(z) = z^d - 2z + 1$ lying inside the unit disk.
\end{lemma}

This is a version of the familiar `gambler's ruin' problem which has been studied extensively. Suppose you start with \$0 and repeatedly play a simple game. Each time you play, you either gain $\$(d-1)$ or lose \$1, each with probability 1/2. The question we seek to answer is this: what is the probability that you will ever have less than \$0 at the conclusion of a game? If the gambler ever drops below \$0, we say that he or she is `ruined'. For a thorough analysis of this problem, see Ethier \cite{ethier}.

\begin{proof} First, note that if $d=1$, each time the game is played, the gambler either loses \$1 or stays even. The only way for the gambler to never drop below his or her initial value is to never lose at all, so the probability of avoiding ruin through the first $N$ games is $2^{-N}$. Clearly in this case the probability of ruin is 1.

In order to handle degrees $d > 1$, we must start with a simplified version of the problem where the gambler is said to `win' if he or she ever reaches a value of at least $\$W$. In this version, the sequence of games ends either when the gambler is ruined (by reaching a value below $\$0$) or wins (by holding a value of at least $\$W$). It is easy to see that the game must end eventually (with either a win or a loss) with probability 1. If the gambler plays enough games, he or she can expect to eventually see every finite subsequence of wins and losses, including $W$ wins in a row (which certainly wins the game, regardless of previous events) and $W$ losses in a row (which certainly loses). Therefore the probability of playing the game forever is zero; eventually the gambler will win or lose. We label $P_{d,W}$ the probability of ruin in a game with upper limit $W$. The probability of ruin in an open-ended game with no upper limit is then $P_d = \lim_{W \rightarrow \infty} P_{d,W}$.

For $k$ in $\ZZ$, let $U_k$ be the probability of ruin (before reaching $\$W$) starting from a value of $\$k$. The value we are trying to compute is $P_{d,W} = U_0$. Clearly $U_k = 1$ for all $k \leq -1$, and $U_k = 0$ for all $k \geq W$. For all other $k$, the values of $U_k$ satisfy a simple recurrence relation:
$$ U_k = \frac{1}{2}U_{k-1} + \frac{1}{2}U_{k+d-1} .$$
The auxiliary polynomial is $g_d(z) = z^d - 2z + 1$. When $d > 2$, this polynomial is separable. But when $d = 2$, the polynomial has a root of multiplicity 2 at $z = 1$, so this must be handled differently. 

First, consider the case $d=2$. In this case, $U_k$ must have the form $U_k = c_1 + c_2k $ for some constants $c_j$. We want to calculate $P_{2,W} = U_0 = c_1$, which we can do by solving a linear system of 2 equations:
\[ \left[ \begin{array}{c c}
1 & -1 \\
1 & W
\end{array} \right] \left[ \begin{array}{c}
c_1 \\ c_2
\end{array} \right] = \left[ \begin{array}{c}
1 \\ 0
\end{array} \right] .\]
We can easily invert the matrix and obtain $P_{2,W} = U_0 = c_1 = \frac{W}{W+1}$. Therefore, the probability of ruin in a game with no upper limit is $P_2 = \lim_{W \rightarrow \infty} U_0 = 1$.

We now move to the case $d > 2$, in which $g_d(z)$ has a root at $\lambda_1 = 1$ and $d-1$ other distinct roots $\lambda_2, \lambda_3, \ldots, \lambda_{d}$. All solutions to the recurrence equation have the form $U_k = \sum_{j=1}^d c_j\lambda_j^k$ for some constants $c_j$. Using the known conditions $U_{-1} = 1$ and $U_W = U_{W+1} = \ldots = U_{W+d-1} = 0$, we can find the needed values of $c_j$ by solving the linear system shown in Figure \ref{fig-f2t-matrix}.

\begin{figure*} \label{fig-f2t-matrix}
\caption{In the $\FF_2[t]$ game, the probability of ruin before reaching a value of $W$ is $\sum_{j=1}^d c_j$.}
\[ \left[ \begin{array}{c c c c c}
\lambda_1^{-1} & \lambda_2^{-1} & \lambda_3^{-1} & \cdots & \lambda_d^{-1} \\
\lambda_1^{W} & \lambda_2^{W} & \lambda_3^{W} & \cdots & \lambda_d^{W} \\
\lambda_1^{W+1} & \lambda_2^{W+1} & \lambda_3^{W+1} & \cdots & \lambda_d^{W+1} \\
\vdots & \vdots & \vdots & & \vdots \\
\lambda_1^{W+d-1} & \lambda_2^{W+d-1} & \lambda_3^{W+d-1} & \cdots & \lambda_d^{W+d-1}
\end{array} \right] \left[ \begin{array}{c}
c_1 \\ c_2 \\ c_3 \\ \vdots \\ c_d
\end{array} \right] = \left[ \begin{array}{c}
1 \\ 0 \\ 0 \\ \vdots \\ 0
\end{array} \right] .\]
\end{figure*}

We then solve this system using Cramer's rule and find the probability of ruin as a function of $W$:
$$ P_{d,W} = U_0 = \sum_{j=1}^d c_j = \frac{ \sum_{j=1}^d (-1)^{1+j} B_j \lambda_j^{-W} }{ \sum_{j=1}^d (-1)^{1+j} B_j \lambda_j^{-1-W} }. $$
The true probability of ruin $P_d$ is the limit of this quantity as $W$ approaches infinity. The dominant term in both the numerator and denominator is the root $\lambda_l$ with the smallest magnitude among the roots of $g_d(z) = z^d - 2z + 1$, assuming there exists a real root inside the unit circle. In fact, it is easy to show\footnote{Details of this proof are given in a supplemental document available on the author's website.} (using Descartes' rule of signs and Rouche's theorem) that $g_d(z)$ must have exactly one root inside the unit circle, and that this root is real and lies in the interval $(1/2,1)$. The value of this root is the probability of ruin $P_d$. \end{proof}

We have now completed the proof of Theorem \ref{th-f2t-terras}. Figure \ref{fig-f2t-psigmatable} shows the values of $P_d$ for $d$ up to 8, accurate to 4 decimal places. Lastly, we prove two simple corollaries following Theorem \ref{th-f2t-terras}.

\begin{figure}
\caption{Finite stopping time probability $P_d$ in $\FF_2[t]$} \label{fig-f2t-psigmatable}
\centering
\[ \begin{array}{|c|c|c|c|c|c|c|c|c|} \hline
d & 1 & 2 & 3 & 4 & 5 & 6 & 7 & 8 \\ \hline
P_d & 1 & 1 & 0.6180 & 0.5437 & 0.5188 & 0.5087 & 0.5041 & 0.5020 \\ \hline
\end{array} \]
\end{figure}

\begin{theorem} \label{cor-f2t-pzero}
If $\deg m \leq 2$, then with the probability that a randomly chosen polynomial will have a divergent $mx+1$ trajectory is zero.
\end{theorem}

\begin{proof}
Let $N = 2^{1+\deg f}$. This is the number of elements of $\FF_2[t]$ of degree $\leq \deg f$. Let $S_0 = \left\{ f \right\}$. With probability 1, there is some $k_1 > 0$ such that $\deg T^{k_1}(f) \leq \deg f$. Without loss of generality, let $k_1$ be the lowest index which satisfies this condition. If $T^{k_1}(f) \in S_0$, then we have returned to a previously visited polynomial and therefore we have found a cycle; otherwise, let $S_1 = S_0 \cup \{ T^{k_1}(f) \}$.

Now with probability 1 there is some minimal $k_2 > k_1$ such that $\deg T^{k_2}(f) \leq \deg T^{k_1}(f) \leq \deg f$. If $T^{k_2}(f) \in S_1$, then we have found a cycle. If not, let $S_2 = S_1 \cup \{ T^{k-2}(f) \}$.

When we iterate the process described above $N$ times, either we find a cycle, or $S_N$ contains every polynomial of degree $\leq \deg f$ (by cardinality). Now with probability 1 there exists $k_{N+1}$ such that $\deg T^{k_{N+1}}(f) \leq \deg f$. This polynomial must have already been visited by the sequence, so this trajectory is a cycle.
\end{proof}

\begin{theorem} \label{cor-f2t-bigsigma}
For any positive integer $N$, we can find a polynomial $f \in \FF_2[t]$ such that $\deg T^k(f) \geq \deg f$ for all $0 < k \leq N$. That is, for any value of $N$, we can find a polynomial whose stopping time is at least $N$.
\end{theorem}

\begin{proof}
Simply create the vector $\left[ 1, 1, \ldots, 1 \right] \in \left\{ 0, 1 \right\}^N$ and use the bijection $\Phi: \FF_2[t]/t^N \rightarrow \left\{ 0, 1 \right\}^N$ to find the polynomial $f$ of degree $<N$ with this parity sequence. The sequence is made entirely of ones, so
$$ \deg T^N(f) = \deg f + N(-1 + \deg m ) .$$
\end{proof}

\subsection{Experimental Results} \label{sec-f2t-experimental}

Our C++ implementation of the $mx+1$ system in $\FF_2[t]$ uses integer arrays to represent elements of $\FF_2[t]$, with each coefficient stored as a single bit. With polynomials represented this way, arithmetic in $\FF_2[t]$ can be programmed entirely using fast bitwise logical operations. Source code for this project can be found at \url{github.com/nichols/polynomial-mxplus1}.

For multiple values of $m \in \FF_2[t]$, we computed the trajectory $\left\{T^k(f)\right\}$ of each polynomial $f \in \FF_2[t]$ of degree $ < 20 $. Each trajectory was computed for $10^5$ steps, or until a cycle was detected (using Brent's cycle-finding algorithm \cite{brent-cycles}). For those polynomials with stopping time $\sigma(f) \leq 10^5$, we recorded the value of $\sigma(f)$; for the rest, we recorded $\sigma(f) > 10^5 $. We conjecture that many if not most of these polynomials have $\sigma(f) = \infty$.

The running time of a single iteration of the $mx+1$ map $T(f)$ is linear in the degree of $f$. Most trajectories tend to either converge quickly to a cycle or else increase linearly in degree indefinitely. For polynomials of the latter type, the running time of computing the first $N$ terms of a trajectory is quadratic in $N$. Accordingly, the small set of apparently divergent trajectories occupied most of the running time of our computations. Figure \ref{fig-f2t-trajectories} shows three different apparently divergent trajectories for $m = t^2 + t + 1$. Notice that all in all three trajectories, the degree appears to increase linearly with slope $2/5$. Every long acyclic trajectory we observed fits this pattern.

\begin{figure*}[h] \label{fig-f2t-trajectories}
\caption{Degree plot of three disjoint trajectories which appear to diverge}
\includegraphics{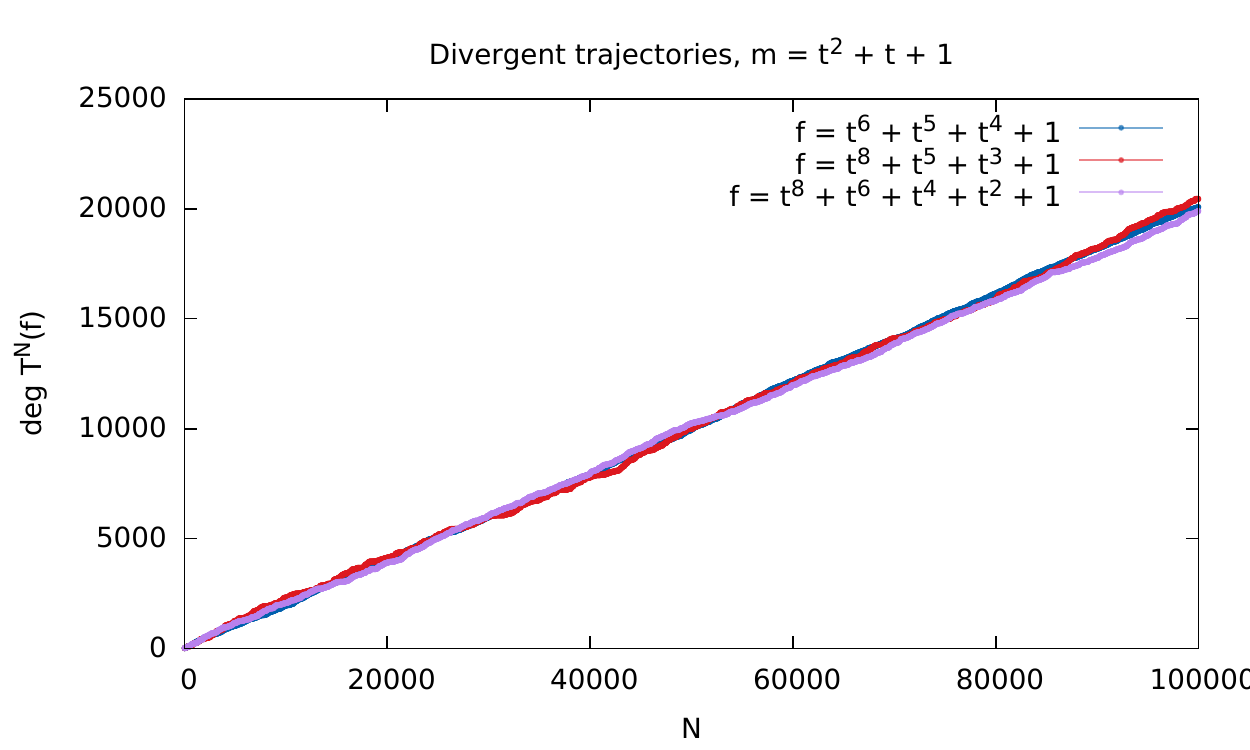}
\end{figure*}

With regard to stopping times, our data supports the theoretical predictions on Theorem \ref{th-f2t-terras} for all the $m$ we tested of degree  not equal to 2. For quadratic $m$, we found a significant number of polynomials with stopping times greater than $10^5$. This does not contradict the theorem's predictions that almost all $f \in \FF_2[t]$ should have finite $mx+1$ stopping time for $\deg m \leq 2$. But it does suggest that the density may converge to zero very slowly.

On the subject of cycle lengths, all the cycles we observed had periods divisible by four, and nearly all were powers of 2. We observed some interesting patterns in the distribution of periods.

\subsubsection{Stopping times}

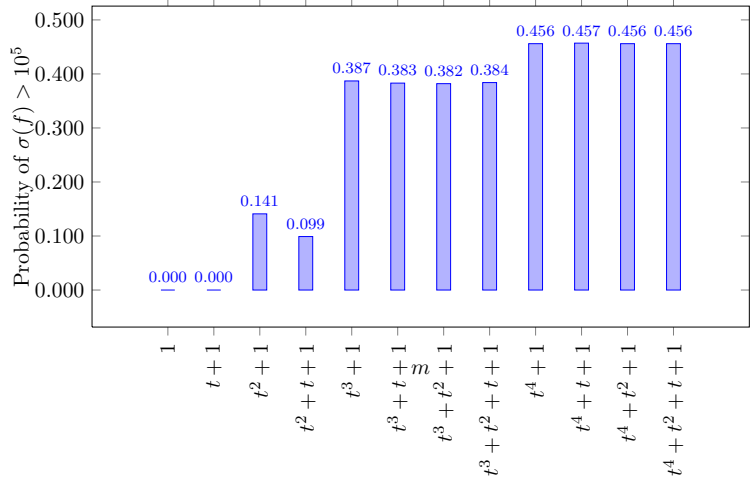
\begin{figure*}[h] \label{fig-f2t-sigma-m}
\caption{Frequency of long stopping times ($\sigma(f) > 10^5$) for various $m \in \FF_2[t]$}
\centering
	\begin{tikzpicture}[scale=0.75]
		\pgfkeys{
		    /pgf/number format/precision=3, 
		    /pgf/number format/fixed zerofill=true,
		    /pgf/number format/fixed
		}
		\begin{axis}[
		x tick label style={
		/pgf/number format/1000 sep=},
		ylabel=Probability of $\sigma(f) > 10^5$,
		enlargelimits=0.15,
		x post scale=1.7,
		legend style={at={(0.5,-0.15)},
		anchor=north,legend columns=-1},
		ybar,
		bar width=7pt,
		symbolic x coords={$1$,$t+1$,$t^2+1$,$t^2+t+1$,$t^3+1$,$t^3+t+1$,$t^3+t^2+1$,$t^3+t^2+t+1$,$t^4+1$,$t^4+t+1$,$t^4+t^2+1$,$t^4+t^2+t+1$},
		xtick=data,
		xlabel=$m$,
		x tick label style={rotate=90,anchor=east},
		nodes near coords,
		every node near coord/.append style={font=\scriptsize},
		nodes near coords align={vertical}
		]
		\addplot
		coordinates { ($1$, 0.000) ($t+1$, 0.000) ($t^2+1$, 0.141) ($t^2+t+1$, 0.099) ($t^3+1$, 0.387) ($t^3+t+1$, 0.383) ($t^3+t^2+1$, 0.382) ($t^3+t^2+t+1$, 0.384) ($t^4+1$, 0.456) ($t^4+t+1$, 0.457) ($t^4+t^2+1$, 0.456) ($t^4+t^2+t+1$, 0.456 ) };
		
		\end{axis}
\end{tikzpicture}
\end{figure*}

Figure \ref{fig-f2t-sigma-m} shows the number of polynomials with $\sigma(f) > 10^5$ for each choice of $m \in \FF_2[t]$. Notice that for $\deg m \neq 2$, we find almost exactly the number of infinite stopping times predicted by Theorem \ref{th-f2t-terras}. When $\deg m = 2$, the theorem predicts that the asymptotic density of infinite stopping time trajectories should be zero, but we found a significant number of polynomials which have stopping times $\sigma(f) > 10^5$. There are two possible explanations for this phenomenon.
\begin{enumerate}
	\item The probability of choosing a polynomial $f$ of degree $< N$ with $\sigma(f) = \infty$ converges to zero very slowly as $N \rightarrow \infty$. That is, there may be many low-degree polynomials with infinite stopping time, but the frequency decreases to zero as the degree increases.
	\item There are a significant number of polynomials which have very high finite stopping times -- in this case, with $\sigma(f) > 10^5$. That is, the distribution of finite stopping times could have a ``long tail''.
\end{enumerate}
To put it another way, Theorem \ref{th-f2t-terras} states that when $\deg m \leq 2$,
$$ P_d = \lim_{N \rightarrow \infty} \left[ \lim_{M \rightarrow \infty} P \left( \sigma(f) > M \;\big|\; \deg f < D \right) \right] = 0 .$$
We found that for both quadratic $m$, this quantity is not close to zero when $M = 10^5$ and $D = 20$, so we would need to increase at least one of these two variables to see evidence of convergence to zero. Figure \ref{fig-f2t-sigma-57} shows the distribution of known stopping times in polynomials of degree $< 20$ for $m = t^2 + 1$ and $m = t^2 + t + 1$, with $m = t + 1$ and $m = t^3 + 1$ presented for comparison. For $m \neq 2$, most trajectories either quickly descend below their initial degree, or apparently diverge. But for quadratic $m$, we see a broader distribution of stopping times. This is yet another reason why the most interesting $mx+1$ systems in $\FF_2[t]$ are those generated by quadratic $m$, and in particular $m = t^2 + t + 1$.

\begin{figure*}[h] \label{fig-f2t-sigma-57}
\caption{Distribution of stopping times among polynomials of degree $< 20$ for four values of $m$.}
\centering
	\[ \begin{array}{|c|c|c|c|c|} \hline
	\sigma(f)	&	t+1	&	t^2+1	&	t^2+t+1	&	t^3+1 \\ \hline
0-50	&	1048573	&	900255	&	930844	&	642494	\\
50-100	&	0	&	413	&	12315	&	0	\\
100-150	&	0	&	0	&	724	&	0	\\
150-200	&	0	&	1	&	90	&	0	\\
200-250	&	0	&	0	&	36	&	0	\\
250-300	&	0	&	0	&	9	&	0	\\
300-350	&	0	&	0	&	5	&	0	\\
350-400	&	0	&	0	&	2	&	0	\\
400-450	&	0	&	0	&	1	&	0	\\ \hline
> 10^5	&	2	&	147906	&	104549	&	406081 \\ \hline
	\end{array}
	\]
\end{figure*}

\subsubsection{Cycle lengths}

We also examined the distribution of periods among $t^2 + t + 1$ trajectories. For $f \in \FF_2[t]$, we define $\lambda(f)$ as follows: if the trajectory of $f$ eventually reaches a cycle of period $N$, then $\lambda(f) = N$. If the trajectory does not become cyclic within $10^5$ iterations, then $\lambda(f) = \infty$. It is clear that $\lambda(f)$ must be even, and with minimal effort one can prove that $\lambda(f) \geq 4$ with equality if and only if $T^N(f) = 1$ for some $N$, because the trivial cycle is the only cycle of length 4.

We observed that nearly all (about 95\%) of $f \in \FF_2[t]$ of degree $< 20$ have $\lambda(f) = 2^k$ for some $k \geq 2$. The highest such $k$ observed was 13. The few remaining trajectories either fail to become cyclic within the first $10^5$ iterations, or become cyclic with periods $\lambda \neq 2^k$. Even in this case, all observed lambdas were multiples of 4.

Within the set of polynomials with $\lambda(f) = 2^k$ for some $k$, we noticed the following pattern: for each $k$, the density of polynomials with $\lambda(f) = 2^k$ appears to increase until it hits a peak, and then gradually tails off. Figure \ref{fig-f2t-lambda} shows for each $k$ the probability that $\lambda(f) = 2^k$ for a randomly chosen $f \in \FF_2[t]$ of degree $d$. So if one selects a random polynomial $f$ of degree $d$, the expected value of $\lambda(f)$ should increase as $d$ increases.

\begin{figure*}[h] \label{fig-f2t-lambda}
\caption{Distribution of $mx+1$ periods for $m = t^2 + t + 1$. Only those polynomials with $\lambda(f)$ a power of 2 are shown.}
\includegraphics[scale=1]{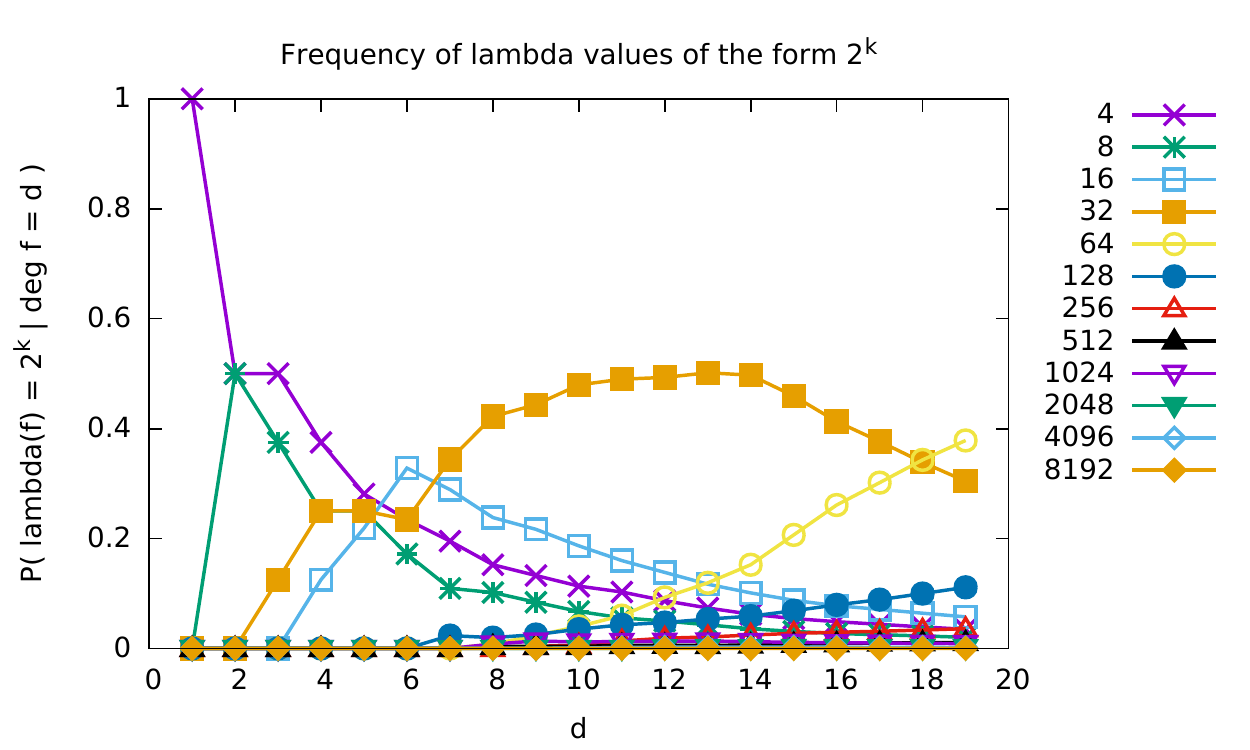}
\end{figure*}

\section{$3x+1$ analogue in the Ring of Functions of an algebraic curve}

In the previous section we investigated $mx+1$ systems in $\FF_2[t]$. As we pointed out in the introduction, $\FF_2[t]$ is the ring of regular functions of the affine line over $F_2$, so it is natural to try to define $mx+1$ systems on rings of functions of other algebraic curves over $\FF_2$. We denote by $R_r$ the ring $\FF_2[x,t] / (x^2 + tx + r(t))$, where $r(t) \in \FF_2[t]$ is some irreducible polynomial.  This is the ring of regular functions on the hyperelliptic curve $x^2 + tx + r(t) = 0$.

Any element $f \in R_r$ has a unique representation of the form $ f(x,t) = f_0(t) + xf_1(t) $ for some $f_0,f_1 \in \FF_2[t]$. Our goal is to define a transformation map $T: R_r \rightarrow R_r$ analogous to the $3x+1$ map in $\ZZ$. We choose a polynomial $m \in R_r$ and define

\[ T(f) = \left\{ \begin{array}{l l}
	\frac{mf + 1 + x}{t}, & f \equiv 1 + x \mod t \\
	\frac{f+x}{t}, & f \equiv x \mod t \\
	\frac{f+1}{t}, & f \equiv 1 \mod t \\
	\frac{f}{t}, & f \equiv 0 \mod t.\\
\end{array} \right. \]

Let $m(x,t) = m_0(t) + xm_1(t)$. Because the ideal $x^2 + tx + r(t)$ is zero, we can write
\begin{align*}
mf + 1 + x &= \left[ m_0f_0 + m_1f_1r + 1 \right] \\
&+ x\left[ m_0f_1 + m_1f_0 + tf_1m_1 + 1 \right] .
\end{align*}
In order to make sure that $mf+1+x$ is always divisible by $t$ when $f \equiv 1+x \mod t$, we require that $m \equiv x \mod t$. 

Repeated iteration of $T$ defines a discrete dynamical system in $R_r$. The trajectory of a given polynomial $f$ is the sequence $T^k(f)$, $k=0,1,2,\ldots $. Each trajectory must either diverge or fall into a cycle (which may be the trivial cycle, $\left\{ 0 \right\}$). There are two parameters that will influence the behavior of the trajectories: the polynomial $r \in \FF_2[t]$ which determines the algebraic curve, and the polynomial $m \in R_r$ used to define the map $T$ on $R_r$. The more interesting of these is $m$, so we will fix $r(t) = t^2 + t + 1$ and study how the dynamics are affected by $m$. As with the $mx+1$ systems in $\FF_2[t]$, we expect that the probability of finding a divergent trajectory will grow with the degree of $m$.

We define the stopping time $\sigma(f)$ to be the minimum number of steps required before the trajectory of $f$ reaches a polynomial of lower degree than $f$. Note that by the `degree' of $f \in R_r$ we always mean the total $t$-degree of $f$, i.e. $\deg f = \max \left\{ \deg f_0, \deg f_1 \right\}$ when $f$ is written as $f_0(t) + xf_1(t)$. Finally, we define the parity sequence of $f$ to be the sequence $p_0, p_1, p_2, \ldots$ where $p_k = (T^k(f))(x,0)$. That is, $p_k \in \left\{ 0, 1, x, 1+x \right\}$ is the congruence class of $T^k(f)$ modulo $t$. We will later use the fact that when $T^k(f) \not\equiv 1 + x \mod t$, $T^{k+1}(f) = ( T^k(f) + p_k ) / t$.

Our ultimate goal is to prove the following analogue of Terras' theorem in this setting.

\begin{theorem} \label{th-f2xt-everett}
For $m \in R_r$ of degree $d$, let $P_d$ be the probability that a randomly chosen polynomial in $R_r$ has finite $mx+1$ stopping time. That is,
$$ P_d = \lim_{N \rightarrow \infty} P \left( \sigma(f) < \infty | \deg f < N \right) .$$
If $d \leq 4$, then $P_d = 1$. If $d > 4$, then $P_d \in (3/4,1)$ is the unique real root of the polynomial $g_d(z) = z^d - 4z + 3$ that lies inside the unit disk.
\end{theorem}

Note that like Theorem \ref{th-f2t-terras}, this is stronger than the analogous result for the integer $3x+1$ system because it provides numerical values for the probability of divergence. Just as in Section 2, our first step is to prove that the parity sequence of a randomly chosen polynomial is distributed uniformly. The parity sequences of all polynomials in the set $\left\{ g + t^N q : q \in R_r \right\}$ must have the same first $N$ terms. Therefore, there is a well-defined function
$$\Phi_m : R_r/t^N \longrightarrow \left\{ 0, 1, x, 1+x \right\}^N $$
which maps each element of $R_r/t^N$ to the first $N$ terms of its parity sequence.

However, we require a special lemma before we can prove that this map is a bijection. Our proof of the analogous result in $\FF_2[t]$ relied on the fact that $\deg fg = \deg f + \deg g$ for all $f,g \in \FF_2[t]$. In $R_r$, we can no longer depend on this assumption, but we can prove a weaker version of this rule by accepting an additional restriction on $m$.

\begin{lemma} \label{th-f2xt-m}
	Let $m = m_0 + xm_1$ and $f = f_0 + xf_1$ be elements of $R_r$. If $\deg m_1 - \deg m_0 < -1$, then $\deg mf = \deg m + \deg f$.
\end{lemma}

For the rest of this paper, when we consider an $mx+1$ system in $R_r$, we always assume $m$ satisfies this condition.

\begin{proof}
	Note that since $\deg m_1 < \deg m_0$, we always have $\deg m = \deg m_0$. Label $g = mf = g_0 + xg_1$. To prove this lemma, we just need to carefully examine the summands of $g_0$ and $g_1$ to determine which has the greatest degree and therefore determines the degree of the sum.
	
	Let $\mu = \deg m_1 - \deg m_0$ and let $\delta = \deg f_1 - \deg f_0$. We must consider three cases:
	\begin{description}
		\item[Case 1: $\delta \leq \mu$.]
		
		Since $\delta \leq \mu < -1$, we know that $\deg f_0 > \deg f_1$, and so the total degree of $f$ is $\deg f = \deg f_0$.
		
		We know that $g_0 = m_0f_0 + rm_1f_1$. Since $\deg m_0 > \deg m_1 + 1$ and $\deg f_0 > \deg f_1 + 1$, we see that $\deg m_0f_0 > \deg m_1f_1r$. Therefore $m_0f_0$ is the dominant term, and so the degree of $g_0$ is $\deg g_0 = \deg m_0 + \deg f_0$.
		
		Now $g_1 = m_0f_1 + m_1f_0 + tm_1f_1$. In this case the dominant term is $m_1f_0$, so $\deg g_1 = \deg m_1 + \deg f_0$. Since $\deg m_0 > \deg m_1$, we have $\deg g_0 > \deg g_1$ and therefore the total degree of $g$ is $\deg g = \deg m_0 + \deg f_0$.
		
		Putting all of this together, we see that $\deg g = \deg m_0 + \deg f_0 = \deg m + \deg f$.
		
		\item[Case 2: $\mu < \delta < -2 - \mu$.]
		
		Recall that $g_0 = m_0f_0 + rm_1f_1$. Using the fact that $\delta < -2 - \mu$, we can see that 
		$$ \deg f_1 - \deg f_0 < -2 - \deg m_1 + \deg m_0 $$
		and therefore
		$$ \deg f_1 + \deg m_1 + 2 < \deg f_0 + \deg m_0 .$$
		
		So the dominant term in $g_0$ is $f_0m_0$, and so $\deg g_0 = \deg f_0 + \deg m_0$.
		
		Next, consider $g_1 = m_0f_1 + m_1f_0 + tm_1f_1$. Since $\delta > \mu$, we have
		$$ \deg f_1 + \deg m_0 > \deg f_0 + \deg m_1 $$
		so the term $f_1m_0$ dominates the term $f_0m_1$. Furthermore, since $\mu < -1$, we have $\deg m_0 > \deg m_1 + 1$, so $f_1m_0$ also dominates $f_1m_1t$. Therefore $\deg g_1 = \deg f_1 + \deg m_0$.
		
		In this case, we don't know whether $\delta$ is positive, negative, or zero. So we can't be sure about which component of $f$ is dominant. However, we have proved that $\deg g_0 = \deg f_0 + \deg m_0$ and that $\deg g_1 = \deg f_1 + \deg m_0$. So either way, $\deg g = \deg f + \deg m$.
		
		\item[Case 3: $\delta \geq -2 - \mu$.]
		
		In this case, $\delta \geq 0$ because $\mu < -1$, so necessarily $\deg f = \deg f_1$. Now consider the degree of $g$. The term $f_1m_0$ in $g_1$ is not dominated by either term of $g_0$. Since $\delta \geq 0$, we know that $\deg f_1 \geq \deg f_0$ and therefore the degree of $f_1m_0$ is not less than the degree of $f_0m_0$. Also, since $\mu < -1$, we know that $\deg m_0 \geq \deg m_1 + 2$, and so the degree of $f_1m_0$ is not less than the degree of $f_1m_1q$. Therefore $\deg g = \deg g_1$. Now we need only find out which term of $g_1$ is dominant.
		
		Because $\mu < -1$, we have $\deg m_0 > \deg m_1 + 1$, so the term $f_1m_0$ dominates $f_1m_1t$. Lastly, since $\mu < 0$ and $\delta \geq 0$, the term $f_1m_0$ dominates $f_0m_1$. Therefore, $f_1m_0$ is the dominant term in $g_1$. In conclusion,
		\begin{align*}
		\deg g &= \deg g_1 \\
		&= \deg f_1 + \deg m_0 \\
		&= \deg f + \deg m
		\end{align*}
		
	\end{description}
	Having proven the desired result in all three cases, we have completed the proof of this lemma.
\end{proof}

Now we are equipped to prove that $\Phi_m$ is a bijection.
\begin{theorem} \label{th-f2xt-phi} The map $\Phi_m$ described above is a set bijection. That is, every sequence $\left\{ p_0, p_1, \ldots, p_{N-1} \right\}$ with $p_i \in \left\{ 0, 1, x, 1+x \right\}$ is the first $N$ terms of the parity sequence of a unique polynomial $f \in R_r$ with $ \deg f < N$. Specifically, the parity sequence determines the initial polynomial $f$ and its $N$-th iterate $f_N$ up to choice of $q_N$:
$$ f = g_{N-1} + t^N q_N ,\quad \deg g_{N-1} < N $$
$$ f_N = h_{N-1} + m^{s(N)}q_N ,\quad \deg h_{N-1} < s(N)\deg m $$
where $s(N)$ is defined
$$s(N) = \#\left\{ 0 \leq k < N : p_k = 1 + x \right\}.$$
\end{theorem}

Note that $s(N)$ is just the number of $1+x$ terms which appear in the first $N$ terms, which is the number of multiplications that occur in the first $N$ steps of the sequence starting from $f$.

As in $\FF_2[t]$, the proof takes the form of an algorithm that yields the unique polynomial in $R_r$ of degree $< N$ with a given parity sequence $\left\{ p_0, p_1, \ldots, p_{N-1} \right\}$. In proving this theorem, we will often be working modulo $t$, and we will frequently use the fact that $m \equiv x \mod t$. Also, we can rewrite the quotient ring $R_r/(t) = \FF_2[x,t]/(x^2+tx+r,t)$ as simply $\FF_2[x,t]/(x^2+1,t)$.

\begin{proof}
We prove the theorem by induction on $N=1,2,\ldots$. First consider $N=1$. There are four cases:
\begin{enumerate}
	\item If $p_0 = 0$, then $f = tq_1$ and $T(f) = q_1$, so $g_0 = 0$.
	\item If $p_0 = 1$, then $f = 1 + tq_1$ and $T(f) = q_1$, so $g_0 = 1$.
	\item If $p_0 = x$, then $f = x + tq_1$ and $T(f) = q_1$, so $g_0 = x$.
	\item If $p_0 = 1+x$, then $f = 1 + x + tq_1$ and $T(f) = (m(1+x) + 1 + x)/t + mq_1$, so $g_0 = 1 + x$.
\end{enumerate}
Each of the above cases gives us a unique $g_0$ from among the elements of $R_r$ of degree $< 1$, as needed. Next, we assume the theorem holds for some $N \geq 1$ and argue that it holds for $N+1$. Let $q_N = tq_{N+1} + v$, meaning $v$ is the element of $\{0, 1, x, 1+x\}$ equivalent to $q_N$ modulo $t$. Here there are just two cases.

\begin{description}

	\item[Case 1: $v = 0$ or $v = 1+x$. ]
	In this case, $m^{s(N)} q_N \equiv x^Z v \equiv v \mod t$, so $ T^N(f) \equiv h_{N-1} + v \mod t$. If $h_{N-1} + v \not\equiv 1+x \mod t$, then
	\begin{align*}
		T^{N+1}(f) &= \frac{ h_{N-1} + m^{s(N)}q_N + p_N }{ t } \\
		&= \frac{ h_{N-1} + m^{s(N)}v + p_N }{ t } + m^{s(N+1)}q_{N+1}.
	\end{align*}
	We now define $h_N = (h_{N-1} + m^{s(N)}v + p_N)/t$.	Referring to Lemma \ref{th-f2xt-m}, we determine that $\deg h_N < s(N)\deg m$, as required.
	
	If instead $h_{N-1} + v \equiv 1 + x \mod t$, then
	\begin{align*}
		T^{N+1}(f) &= \frac{ m \left( h_{N-1} + m^{s(N)}q_N \right) + 1 + x }{ t } \\
		&= \frac{ mh_{N-1} + m^{s(N+1)}v + 1 + x }{ t } + m^{s(N+1)}q_{N+1}.
	\end{align*}
	Once again we see that the degree of $ h_N = ( mh_{N-1} + m^{s(N+1)}v + 1 + x ) / t$ satisfies the condition of the theorem.
	
	\item[Case 2: $v = 1$ or $v = x$.]
	In this case,
	\[ m^{s(N)} q_N \equiv x^{s(N)} v \equiv \left\{ \begin{array}{l l}
	v, & s(N)\trm{ even} \\
	xv, & s(N)\trm{ odd}
	\end{array} \right. \mod t. \]
	So in order to make $T^N(f) = h_{N-1} + m^{s(N)}q_N$ be equivalent to $1+x \mod t$, one of the following must be true:
	\begin{itemize}
		\item $s(N)$ even, $h_{N-1} + v \equiv 1+x \mod t$
		\item $s(N)$ odd, $h_{N-1} + xv \equiv 1+x \mod t$.
	\end{itemize}
	If so, then
	\begin{align*}
		T^{N+1}(f) &= \frac{ m T^N(f) + 1 + x }{ t } \\
		&= \frac{ mh_{N-1} + m^{s(N+1)}v + 1 + x }{ t } + m^{s(N+1)}q_{N+1}
	\end{align*}
	and we define $h_N = ( mh_{N-1} + m^{s(N+1)}v + 1 + x )/t$, which has degree $< s(N+1)\deg m$.
	
	If neither of those two possibilities occurs, then
	\begin{align*}
		T^{N+1}(f) &= \frac{ h_{N-1} + m^{s(N)}( v + tq_N ) + p_N }{ t } \\
		&= \frac{ h_{N-1} + m^{s(N+1)}v + p_N }{ t } + m^{S(N+1)}q_N.
	\end{align*}
	So $h_N = ( h_{N-1} + m^{s(N)}v + p_N )/t$ satisfies $\deg h_N < s(N+1)\deg m$ as required.
	
\end{description}

We have established that a vector $\vec{p} = ( p_0, p_1, \ldots, p_{N-1} ) \in \left\{ 0, 1, x, t + x \right\}^N$ determines a unique polynomial $g_{N-1} \in R_r$ of degree $< N$ such that $\vec{p}$ is the first $N$ terms of the parity sequence of $f$, and that all polynomials in $R_r$ that satisfy this are of the form $g_{N-1} + t^Nq_N$ for some $q_N$. There are $4^N$ polynomials of degree $<N$ and there are $4^N$ elements of $\left\{ 0, 1, x, 1+x \right\}^N$. So by cardinality, the surjective map $\Phi$ is a bijection.
\end{proof}

We have now shown that the parity sequence of a randomly chosen $f \in R_r$ of degree $< N$ is distributed uniformly in $\left\{ 0, 1, x, 1+x \right\}^N$. Following the same outline as the $\FF_2[t]$ proof, we can then model the degree of $T^k(f)$ as a random walk:
$$ \deg T^k(f) = \deg f - N + (\deg m)\sum_{k=0}^{N-1}X_k $$
where $X_k$ are IID Bernoulli random variables. The difference is that this time, $X_k$ takes the value 1 with probability 1/4 and 0 otherwise. So the probability that a randomly chosen $f \in R_r$ has finite stopping time is equal to
$$ P \left( \exists N > 0 : \sum_{k=0}^{N-1}X_k < \frac{N}{d} \right) $$
where $d = \deg m$. This is just another version of the gambler's ruin problem, so we can prove the following result using the same methods as in $\FF_2[t]$.

\begin{lemma}\label{th-f2xt-ruin}
For $d > 0$, let $P_d$ be defined
$$ P_d = P \left( \exists N > 0 : \sum_{k=0}^{N-1}X_k < \frac{N}{d} \right) $$
where $X_i$ are IID Bernoulli variables taking the value 1 with probability 1/4 and 0 otherwise. If $d \leq 4$, then $P_d = 1$. If $d > 4$, then $P_d$ is the unique root of $g_d(z) = z^d - 4z + 3$ inside the unit disk, which is real and lies in the interval $(3/4,1)$. 
\end{lemma}

This time, the gambler repeatedly plays a game which pays out $d - 1$ dollars with probability $1/4$, and $-1$ dollars with probability $3/4$. The stopping time corresponds to the number of games before the gambler goes broke. The proof is essentially the same as that of the analogous result in $\FF_2[t]$. In this case the linear recurrence relation is
$$ U_k = \frac{3}{4}U_{k-1} + \frac{1}{4}U_{k+d-1} $$
and our goal is to find the value of $U_0$, representing the probability of ruin (depending on $W$) starting from a value of 0. As in Section \ref{sec-proof-f2t}, we solve the system using Cramer's rule and then take the limit of this quantity as $W \rightarrow \infty$ to find the probability of ruin in a game with no upper limit.\footnote{Full details of this proof are given in a supplemental document available on the author's website.} Figure \ref{fig-f2xt-psigmatable} shows the probability of finite stopping time for $m$ of degree up to 8, accurate to 4 decimal places.

\begin{figure*}[h]
\caption{Finite stopping time probability $P_d$ in $R_r$} \label{fig-f2xt-psigmatable}
\centering
\[
\begin{array}{|c|c|c|c|c|c|c|c|c|} \hline
d & 1 & 2 & 3 & 4 & 5 & 6 & 7 & 8 \\ \hline
P_d & 1	& 1 & 1 & 1	& 0.8882 & 0.8343 & 0.8046 & 0.7867 \\ \hline
\end{array} \]
\end{figure*}

Once again, we prove some corollaries of this result.

\begin{theorem}
Let $d \leq 4$ and let $m \in R_r$ of degree $d$. Then a randomly chosen polynomial $f \in R_r$ has finite $mx+1$ stopping time with probability 1.
\end{theorem}

We use exactly the same proof as in Theorem \ref{cor-f2t-pzero}, with the minor difference that $N = 4^{1 + \deg f}$ instead of $2^{1 + \deg f}$.

\begin{theorem}
For any positive integer $N$, we can find a polynomial $f \in R_r$ such that $\sigma(f) > N$.
\end{theorem}

\begin{proof}
Simply create the vector $\left[ 1+x, 1+x, \ldots, 1+x \right] \in \left\{ 0, 1, x, 1+x \right\}^N$ and use the bijection $\Phi_m: R_r/t^N \rightarrow \left\{ 0, 1 x, 1+x \right\}^N$ to find the polynomial $f$ of degree $<N$ with this parity sequence. The sequence is made entirely of ones, so
$$ \deg T^N(f) = \deg f + N(-1 + \deg m ) $$
(once again we rely on Lemma \ref{th-f2xt-m}).
\end{proof}

\subsection{Experimental Data}

As in $\FF_2[t]$, we implemented the $mx+1$ system in $R_r$ in such a way as to make computations as efficient as possible. For each polynomial $f = f_0(t) + xf_1(t) \in R_r$ with $\deg f < 10$, we computed the trajectory of $f$ up to $10^5$ iterations of the $mx+1$ function. We carried out this process for several choices of $m = m_0 + xm_1$ with $m_0, m_1 \in \FF_2[t]$. Figure \ref{fig-f2xt-timeouts} shows the density of polynomials with long stopping times for each $m$. Much like the $\FF_2[t]$ case, the data generally agrees with our predictions, though we do see a higher than expected occurrence of high stopping times when the degree of $m$ is a particular value. In $R_r$, the most interesting $m$ polynomials seem to be those of degree 4.

\begin{figure*}[h]\label{fig-f2xt-timeouts}
\centering
\caption{Density of long stopping times ($\sigma(f) > 10^5$) for various $m \in R_r$.}
	\begin{tikzpicture}[scale=2]

	\draw [fill, color=blue, opacity=0.028 ] ( 0, 0 ) rectangle ( 0.4, 0.4 );
	\node at (0.2, 0.2) {\small 0.03};
	\draw [fill, color=blue, opacity=0.004 ] ( 0, 0.4 ) rectangle ( 0.4, 0.8 );
	\node at (0.2, 0.6) {\small 0.00};
	\node [below, anchor=west, rotate=-60] at (0.2, 0) { $t^3$ };
	\draw [fill, color=blue, opacity=0.031 ] ( 0.4, 0 ) rectangle ( 0.8, 0.4 );
	\node at (0.6, 0.2) {\small 0.03};
	\draw [fill, color=blue, opacity=0.007 ] ( 0.4, 0.4 ) rectangle ( 0.8, 0.8 );
	\node at (0.6, 0.6) {\small 0.01};
	\node [below, anchor=west, rotate=-60] at (0.6, 0) { $t^3+t$ };
	\draw [fill, color=blue, opacity=0.008 ] ( 0.8, 0 ) rectangle ( 1.2, 0.4 );
	\node at (1, 0.2) {\small 0.01};
	\draw [fill, color=blue, opacity=0.009 ] ( 0.8, 0.4 ) rectangle ( 1.2, 0.8 );
	\node at (1, 0.6) {\small 0.01};
	\node [below, anchor=west, rotate=-60] at (1, 0) { $t^3+t^2$ };
	\draw [fill, color=blue, opacity=0.008 ] ( 1.2, 0 ) rectangle ( 1.6, 0.4 );
	\node at (1.4, 0.2) {\small 0.01};
	\draw [fill, color=blue, opacity=0.009 ] ( 1.2, 0.4 ) rectangle ( 1.6, 0.8 );
	\node at (1.4, 0.6) {\small 0.01};
	\node [below, anchor=west, rotate=-60] at (1.4, 0) { $t^3+t$ };
	\draw [fill, color=blue, opacity=0.105 ] ( 1.6, 0 ) rectangle ( 2, 0.4 );
	\node at (1.8, 0.2) {\small 0.10};
	\draw [fill, color=blue, opacity=0.073 ] ( 1.6, 0.4 ) rectangle ( 2, 0.8 );
	\node at (1.8, 0.6) {\small 0.07};
	\draw [fill, color=blue, opacity=0.036 ] ( 1.6, 0.8 ) rectangle ( 2, 1.2 );
	\node at (1.8, 1) {\small 0.04};
	\draw [fill, color=blue, opacity=0.036 ] ( 1.6, 1.2 ) rectangle ( 2, 1.6 );
	\node at (1.8, 1.4) {\small 0.04};
	\node [below, anchor=west, rotate=-60] at (1.8, 0) { $t^4$ };
	\draw [fill, color=blue, opacity=0.108 ] ( 2, 0 ) rectangle ( 2.4, 0.4 );
	\node at (2.2, 0.2) {\small 0.11};
	\draw [fill, color=blue, opacity=0.070 ] ( 2, 0.4 ) rectangle ( 2.4, 0.8 );
	\node at (2.2, 0.6) {\small 0.07};
	\draw [fill, color=blue, opacity=0.036 ] ( 2, 0.8 ) rectangle ( 2.4, 1.2 );
	\node at (2.2, 1) {\small 0.04};
	\draw [fill, color=blue, opacity=0.037 ] ( 2, 1.2 ) rectangle ( 2.4, 1.6 );
	\node at (2.2, 1.4) {\small 0.04};
	\node [below, anchor=west, rotate=-60] at (2.2, 0) { $t^4+t$ };
	\draw [fill, color=blue, opacity=0.075 ] ( 2.4, 0 ) rectangle ( 2.8, 0.4 );
	\node at (2.6, 0.2) {\small 0.07};
	\draw [fill, color=blue, opacity=0.074 ] ( 2.4, 0.4 ) rectangle ( 2.8, 0.8 );
	\node at (2.6, 0.6) {\small 0.07};
	\draw [fill, color=blue, opacity=0.037 ] ( 2.4, 0.8 ) rectangle ( 2.8, 1.2 );
	\node at (2.6, 1) {\small 0.04};
	\draw [fill, color=blue, opacity=0.037 ] ( 2.4, 1.2 ) rectangle ( 2.8, 1.6 );
	\node at (2.6, 1.4) {\small 0.04};
	\node [below, anchor=west, rotate=-60] at (2.6, 0) { $t^4+t^2$ };
	\draw [fill, color=blue, opacity=0.076 ] ( 2.8, 0 ) rectangle ( 3.2, 0.4 );
	\node at (3, 0.2) {\small 0.08};
	\draw [fill, color=blue, opacity=0.071 ] ( 2.8, 0.4 ) rectangle ( 3.2, 0.8 );
	\node at (3, 0.6) {\small 0.07};
	\draw [fill, color=blue, opacity=0.038 ] ( 2.8, 0.8 ) rectangle ( 3.2, 1.2 );
	\node at (3, 1) {\small 0.04};
	\draw [fill, color=blue, opacity=0.036 ] ( 2.8, 1.2 ) rectangle ( 3.2, 1.6 );
	\node at (3, 1.4) {\small 0.04};
	\node [below, anchor=west, rotate=-60] at (3, 0) { $t^4+t$ };
	\draw [fill, color=blue, opacity=0.046 ] ( 3.2, 0 ) rectangle ( 3.6, 0.4 );
	\node at (3.4, 0.2) {\small 0.05};
	\draw [fill, color=blue, opacity=0.040 ] ( 3.2, 0.4 ) rectangle ( 3.6, 0.8 );
	\node at (3.4, 0.6) {\small 0.04};
	\draw [fill, color=blue, opacity=0.036 ] ( 3.2, 0.8 ) rectangle ( 3.6, 1.2 );
	\node at (3.4, 1) {\small 0.04};
	\draw [fill, color=blue, opacity=0.037 ] ( 3.2, 1.2 ) rectangle ( 3.6, 1.6 );
	\node at (3.4, 1.4) {\small 0.04};
	\node [below, anchor=west, rotate=-60] at (3.4, 0) { $t^4+t^3$ };
	\draw [fill, color=blue, opacity=0.044 ] ( 3.6, 0 ) rectangle ( 4, 0.4 );
	\node at (3.8, 0.2) {\small 0.04};
	\draw [fill, color=blue, opacity=0.037 ] ( 3.6, 0.4 ) rectangle ( 4, 0.8 );
	\node at (3.8, 0.6) {\small 0.04};
	\draw [fill, color=blue, opacity=0.037 ] ( 3.6, 0.8 ) rectangle ( 4, 1.2 );
	\node at (3.8, 1) {\small 0.04};
	\draw [fill, color=blue, opacity=0.036 ] ( 3.6, 1.2 ) rectangle ( 4, 1.6 );
	\node at (3.8, 1.4) {\small 0.04};
	\node [below, anchor=west, rotate=-60] at (3.8, 0) { $t^4+t^3+t$ };
	\draw [fill, color=blue, opacity=0.046 ] ( 4, 0 ) rectangle ( 4.4, 0.4 );
	\node at (4.2, 0.2) {\small 0.05};
	\draw [fill, color=blue, opacity=0.039 ] ( 4, 0.4 ) rectangle ( 4.4, 0.8 );
	\node at (4.2, 0.6) {\small 0.04};
	\draw [fill, color=blue, opacity=0.036 ] ( 4, 0.8 ) rectangle ( 4.4, 1.2 );
	\node at (4.2, 1) {\small 0.04};
	\draw [fill, color=blue, opacity=0.038 ] ( 4, 1.2 ) rectangle ( 4.4, 1.6 );
	\node at (4.2, 1.4) {\small 0.04};
	\node [below, anchor=west, rotate=-60] at (4.2, 0) { $t^4+t^3+t^2$ };
	\draw [fill, color=blue, opacity=0.044 ] ( 4.4, 0 ) rectangle ( 4.8, 0.4 );
	\node at (4.6, 0.2) {\small 0.04};
	\draw [fill, color=blue, opacity=0.036 ] ( 4.4, 0.4 ) rectangle ( 4.8, 0.8 );
	\node at (4.6, 0.6) {\small 0.04};
	\draw [fill, color=blue, opacity=0.037 ] ( 4.4, 0.8 ) rectangle ( 4.8, 1.2 );
	\node at (4.6, 1) {\small 0.04};
	\draw [fill, color=blue, opacity=0.035 ] ( 4.4, 1.2 ) rectangle ( 4.8, 1.6 );
	\node at (4.6, 1.4) {\small 0.04};
	\node [below, anchor=west, rotate=-60] at (4.6, 0) { $t^4+t^3+t$ };
	\draw [fill, color=blue, opacity=0.153 ] ( 4.8, 0 ) rectangle ( 5.2, 0.4 );
	\node at (5, 0.2) {\small 0.15};
	\draw [fill, color=blue, opacity=0.145 ] ( 4.8, 0.4 ) rectangle ( 5.2, 0.8 );
	\node at (5, 0.6) {\small 0.14};
	\draw [fill, color=blue, opacity=0.127 ] ( 4.8, 0.8 ) rectangle ( 5.2, 1.2 );
	\node at (5, 1) {\small 0.13};
	\draw [fill, color=blue, opacity=0.128 ] ( 4.8, 1.2 ) rectangle ( 5.2, 1.6 );
	\node at (5, 1.4) {\small 0.13};
	\node [below, anchor=west, rotate=-60] at (5, 0) { $t^5$ };
	\draw (0, 0) -- (5.2, 0) -- (5.2, 1.6) -- (0, 1.6) -- cycle;
	\node [left, anchor=east] at (0, 0.2) { $1$ };
	\node [left, anchor=east] at (0, 0.6) { $t+1$ };
	\node [left, anchor=east] at (0, 1) { $t^2 + 1$ };
	\node [left, anchor=east] at (0, 1.4) { $t^2 + t + 1$ };
	\node [below] at ( 2.6, -1 ) {$m_0$};
	\node [left] at ( -1, 0.8 ) {$m_1$};

\end{tikzpicture}
\end{figure*}

\end{document}